\documentclass[a4paper,12pt]{article}
\usepackage[utf8]{inputenc}
\usepackage[UKenglish]{isodate}
\cleanlookdateon
\usepackage{amssymb}
\usepackage{amsmath}
\usepackage{amsthm}
\usepackage{enumitem}
\usepackage[T1]{fontenc}
\usepackage{xcolor}
\usepackage[hyphens]{url} 
\usepackage[linktoc = all, hidelinks, colorlinks, unicode=true]{hyperref} 
\usepackage[capitalise, compress, nameinlink, noabbrev]{cleveref} 
\usepackage{esvect}
\usepackage[sort,compress,numbers]{natbib}
\newcommand{\tabitem}{~~\llap{\textbullet}~~}
\usepackage{tcolorbox}

\hypersetup{
  colorlinks   = true,
  urlcolor     = blue!50!black,
  linkcolor    = blue!50!black,
  citecolor   = blue!50!black
}

\title{\texorpdfstring{\vspace{-4ex}}{}Invertibility of digraphs and tournaments}
\author{Noga Alon\protect\footnote{Department of Mathematics, Princeton University, Princeton, NJ, USA.
Research supported in part by NSF grant DMS-2154082
and BSF grant 2018267.} \quad
Emil Powierski\protect\footnote{Mathematical Institute, University of Oxford,
Oxford OX2 6GG, UK.} \quad
Michael Savery\protect\footnotemark[2]~\footnote{Heilbronn Institute for Mathematical Research, Bristol, UK.} \\
Alex Scott\protect\footnotemark[2]~\footnote{Research supported by EPSRC grant EP/X013642/1.} \quad
Elizabeth Wilmer\protect\footnote{Department of Mathematics, Oberlin College, Oberlin, OH, 44074, USA.}}
\date{15 September 2023}
\newcommand\blfootnote[1]{
  \begingroup
  \renewcommand\thefootnote{}\footnote{#1}
  \addtocounter{footnote}{-1}
  \endgroup
}

\newcommand{\floor}[1]{\left \lfloor{ #1 }\right \rfloor}
\newcommand{\ceil}[1]{\left \lceil{ #1 }\right \rceil}
\newcommand{\abs}[1]{\lvert#1\rvert}
\newcommand{\mbf}[1]{\mathbf{ #1 }}
\newcommand{\inv}{\operatorname{inv}}
\newcommand{\rk}{\operatorname{rank}}
\newcommand{\ra}{\rightarrow}
\newcommand{\N}{\mathbb{N}}

\newcommand{\F}{\mathbb{F}}
\newcommand{\cU}{\mathcal{U}}
\newcommand{\cB}{\mathcal{B}}

\newcommand{\hD}{\widehat{D}}
\newcommand{\bP}{\mathbb{P}}
\newcommand{\cM}{\mathcal{M}}
\newcommand{\cW}{\mathcal{W}}

\newtheorem{theorem}{Theorem}
\newtheorem{lemma}[theorem]{Lemma}
\newtheorem{corollary}[theorem]{Corollary}
\newtheorem{prop}[theorem]{Proposition}
\newtheorem{claim}{Claim}
\newtheorem{observation}{Observation}
\newtheorem{conjecture}{Conjecture}
\newtheorem{question}[conjecture]{Question}
\newtheorem{problem}[conjecture]{Problem}
\theoremstyle{definition}

\newtheorem{definition}{Definition}

\advance\textwidth2\evensidemargin
\begin{document}

\maketitle

\begin{abstract}
\noindent For an oriented graph $D$ and a set $X\subseteq V(D)$, the \emph{inversion of $X$ in $D$} is the digraph obtained by reversing the orientations of the edges of $D$ with both endpoints in $X$. The \emph{inversion number of $D$}, $\inv(D)$, is the minimum number of inversions which can be applied in turn to $D$ to produce an acyclic digraph. Answering a recent question of Bang-Jensen, da Silva, and Havet we show that, for each $k\in\N$ and tournament $T$, the problem of deciding whether $\inv(T)\leq k$ is solvable in time $O_k(\abs{V(T)}^2)$, which is tight for all $k$. In particular, the problem is fixed-parameter tractable when parameterised by $k$. On the other hand, we build on their work to prove their conjecture that for $k\geq 1$ the problem of deciding whether a general oriented graph $D$ has $\inv(D)\leq k$ is NP-complete. We also construct oriented graphs with inversion number equal to twice their cycle transversal number, confirming another conjecture of Bang-Jensen, da Silva, and Havet, and we provide a counterexample to their conjecture concerning the inversion number of so-called `dijoin' digraphs while proving that it holds in certain cases.  Finally, we asymptotically solve the natural extremal question in this setting, improving on previous bounds of Belkhechine, Bouaziz, Boudabbous, and Pouzet to show that the maximum inversion number of an $n$-vertex tournament is $(1+o(1))n$.\blfootnote{Email: \textsf{\href{mailto:nalon@math.princeton.edu}{nalon@math.princeton.edu}, \{\href{mailto:powierski@maths.ox.ac.uk}{powierski},\href{mailto:savery@maths.ox.ac.uk}{savery},\href{mailto:scott@maths.ox.ac.uk}{scott}\}@maths.ox.ac.uk, \href{mailto:ewilmer@oberlin.edu}{ewilmer@oberlin.edu}}.}
\end{abstract}

\section{Introduction}\label{sec:intro}

In this paper we only consider digraphs without loops, digons, or parallel edges, for which we use the terms \emph{digraph} and \emph{oriented graph} interchangeably. For such a digraph $D=(V,E)$ and a set $X\subseteq V$, the \emph{inversion of $X$ in $D$} is the digraph obtained from $D$ by reversing the direction of the edges with both endpoints in $X$; we say that we \emph{invert $X$ in $D$}. Given a family of sets $X_1,\dots,X_k\subseteq V$, we can invert $X_1$ in $D$, then $X_2$ in the resulting digraph, and so on, noting that the final digraph produced by these inversions is independent of the order in which we perform them. If inverting $X_1,\dots,X_k$ in turn transforms $D$ into an acyclic digraph, then we say that these sets form a \emph{decycling family} of $D$. We will refer to a set $X\subseteq V$ which forms a decycling family by itself as a \emph{decycling set}. The \emph{inversion number of $D$}, denoted $\inv(D)$, is defined to be the minimum size of a decycling family of $D$, and for $k\in\N_0$ we say that~$D$ is \emph{$k$-invertible} if $\inv(D)\leq k$.

The study of inversions began in Houmem Belkhechine's PhD thesis~\cite{belkhechine:thesis} and continued in~\cite{BELKHECHINE_unpublished, BELKHECHINE2010703, pouzet}, in which many foundational results were established. The present work is inspired by a recent paper of Bang-Jensen, da Silva, and Havet~\cite{bangjensen2022inversion} which studied a wide range of questions about invertibility, with an emphasis on those of an algorithmic or extremal nature. They also posed a host of interesting conjectures and problems, some of which we answer in this paper.

\subsection{The inversion number of \texorpdfstring{$k$-joins}{k-joins}}\label{subsec:intro_kjoins}

The cornerstone of many of the conjectures made by Bang-Jensen, da Silva, and Havet in~\cite{bangjensen2022inversion} is the following `dijoin conjecture'. For oriented graphs $L$ and $R$, the \emph{dijoin $L\ra R$ from $L$ to $R$} is the oriented graph consisting of vertex-disjoint copies of $L$ and $R$, with an edge $\vv{uv}$ for all $u\in V(L)$ and $v\in V(R)$. 
\begin{conjecture}[\cite{bangjensen2022inversion}]\label{conj:dijoin}
For oriented graphs $L$ and $R$ we have $\inv(L\ra R)=\inv(L)+\inv(R)$.
\end{conjecture}

Noting that the left-hand side is certainly at most the right-hand side for all $L$ and~$R$ and that the conjecture holds trivially if $\inv(L)=0$ or $\inv(R)=0$, Bang-Jensen, da Silva, and Havet showed it to be true when $\inv(L)+\inv(R)\leq 3$, and when $\inv(L)=\inv(R)=2$ and $L$ and $R$ are both strongly connected.\footnote{The case where $\inv(L)=2$ and $\inv(R)=1$ is not explicitly mentioned in~\cite{bangjensen2022inversion}, but follows easily from the case where $\inv(L)=1$ and $\inv(R)=2$ by inverting $V(L\ra R)$.} They also observed (see our \cref{sec:dijoin}) that the conjecture is equivalent to its restriction to tournaments. We disprove \cref{conj:dijoin} by exhibiting a tournament $R$ with $\inv(R)=\inv(\vv{C_3}\ra R)$, where $\vv{C_3}$ is the directed cycle on three vertices.

\begin{theorem}\label{thm:dijoin_ce}
There exists a tournament $R$ with $\inv(R)=\inv(\vv{C_3}\ra R)=3$.
\end{theorem}

While \cref{thm:dijoin_ce} shows that the dijoin conjecture is false in general, we prove it in the case where $\inv(L)=\inv(R)=2$.

\begin{theorem}\label{thm:2ra2}
If $L$ and $R$ are digraphs with $\inv(L)=\inv(R)=2$, then $\inv(L\ra R)=4$.
\end{theorem}

The proof of \cref{thm:2ra2} relies on the strongly connected case and our next result, which concerns the following generalisation of dijoins to arbitrarily many digraphs. For $k\in \N$ the \emph{$k$-join} of digraphs $D_1,\dots,D_k$, written $[D_1, \dots, D_k]$, is the digraph consisting of vertex-disjoint copies of $D_1, \dots, D_k$ with an additional edge $\vv{uv}$ whenever $u \in V(D_i), v \in V(D_j)$ for $i<j$. We write $[D]_k=[D, \dots, D]$ for the $k$-join of $k$ copies of the same oriented graph~$D$. The following result can be viewed as a $k$-join analogue of the dijoin conjecture holding under certain conditions. It generalises a theorem of Pouzet, Kaddour, and Thatte~\cite{pouzet} which states that $\inv([\vv{C_3}]_k)=k$ for all $k$.

\begin{theorem}\label{thm:inv_kjoin}
Let $k\in\N$ and let $D_1,\dots,D_k$ be oriented graphs. Assume that\/ $\inv(D_i)\leq 2$ for all $i$, with equality for at most one $i$. Then
\begin{align}\label{eq:thm_nested}
\inv([D_1,\dots,D_k])=\sum_{i=1}^k\inv(D_i).
\end{align}
\end{theorem}

We will use \cref{thm:inv_kjoin} to confirm another conjecture from~\cite{bangjensen2022inversion} which was made based on the dijoin conjecture (see \cref{thm:NPcomplete} below). \Cref{thm:inv_kjoin} and, in turn, \cref{thm:2ra2} follow from a characterisation of the decycling families of size $k$ of arbitrary $k$-joins of oriented graphs with inversion number $1$. We will need some further terminology to state this result: for a digraph $D$, sets $X_1,\dots,X_k\subseteq V(D)$, and a vertex $v\in V(D)$, we define the \emph{characteristic vector} of $v$ in $X_1,\dots,X_k$ to be $(I_{\{v\in X_i\}}:i\in[k])\in\F_2^k$, where $I_{\{v\in X_i\}}$ is the indicator function of the event $v\in X_i$. For vectors $\mbf{u},\mbf{v}\in \F_2^k$ we write $\mbf{u}\cdot\mbf{v}$ for the usual scalar product of $\mbf{u}$ and $\mbf{v}$ over $\F_2$. This is not a genuine inner product, but we say nevertheless that a collection $\mbf{u}_1, \dots, \mbf{u}_\ell \in \F_2^k$ is \emph{orthonormal} if $\mbf{u}_i\cdot\mbf{u}_i=1$ for all $i$ and $\mbf{u}_i\cdot\mbf{u}_j=0$ for all $i\neq j$.
Finally, we refer to the canonical copy of $D_i$ in $D=[D_1, \dots, D_k]$ as the \emph{$i$th factor} of $D$.
We are now ready to state our characterisation theorem, the case $k=2$ of which was shown by Bang-Jensen, da Silva, and Havet~\cite{bangjensen2022inversion}. Its proof is based on an approach used by Pouzet, Kaddour, and Thatte~\cite{pouzet}.

\begin{theorem}\label{thm:kjoin_char}
Let $D_1,\dots,D_k$ be oriented graphs with $\inv(D_i)=1$ for all $i$ and let $\hD=[D_1, \dots, D_k]$ be their $k$-join. Then sets $X_1, \dots X_k\subseteq V(\hD)$ form a decycling family of $\hD$ if and only if there are orthonormal vectors
$\mbf{u}_1,\dots,\mbf{u}_k\in\F_2^k$ and for each $i$ a decycling set $V_i\subseteq V(D_i)$ of the $i$th factor of $\hD$ such that, for each $i$, the vertices in $V_i$ have characteristic vector $\mbf{u}_i$ \textnormal{(}in $X_1 \dots, X_k$\textnormal{)}, and all other vertices have characteristic vector $\mbf{0}$ \textnormal{(}in $X_1 \dots, X_k$\textnormal{)}.

In particular, any acyclic digraph obtained from $\hD$ by $k$ inversions can also be obtained by inverting a decycling set for each factor in turn.
\end{theorem}

\subsection{Computational complexity}\label{subsec:intro_complexity}

One focus of Bang-Jensen, da Silva, and Havet's paper~\cite{bangjensen2022inversion} was on the computational complexity of deciding whether an oriented graph is $k$-invertible. More formally, they considered, for fixed $k\in\N_0$, the problem of \textsc{$k$-Inversion}:

\begin{center}
\tcbox[
colback=white, colframe=black]{
\begin{tabular}{l}
\textsc{\underline{Input}:} an oriented graph $D$. \\
\textsc{\underline{Problem}:} $\inv(D)\leq k$?
\end{tabular}
}
\end{center}

\noindent A first observation is that \textsc{$0$-Inversion} is equivalent to checking whether a digraph $D$ is acyclic, which is well known to be possible in time $O(\abs{V(D)}^2)$ (see~\cite[p. 612]{introtoalgs}), so we need only consider $k\geq 1$.

Bang-Jensen, da Silva, and Havet~\cite{bangjensen2022inversion} showed that \textsc{$1$-Inversion} is NP-complete using a reduction from \textsc{Monotone 3-in-1 SAT}. Then, using the special cases of the dijoin conjecture proved in that paper, they observed that for a digraph $D$ we have $\inv(D\ra D)=2$ if and only if $\inv(D)=1$, from which it follows that \textsc{$2$-Inversion} is also NP-complete. They conjectured that NP-completeness extends to \textsc{$k$-Inversion} for all $k\geq 3$, noting that this would follow from a similar argument if the dijoin conjecture were true. Of course, the full dijoin conjecture is not required, and indeed it is easy to see that \cref{thm:inv_kjoin} is enough: it implies that $\inv(D)=1$ if and only if $\inv([D]_k)=k$, which reduces \textsc{$1$-Inversion} to \textsc{$k$-Inversion} and hence shows the following (see \cref{sec:kjoins}).

\begin{theorem}\label{thm:NPcomplete}
\textsc{$k$-Inversion} is NP-complete for all $k\in\N$.
\end{theorem}

Bang-Jensen, da Silva, and Havet also considered the computational complexity of the same problem when the input is restricted to tournaments. For fixed $k\in\N$ the problem of \textsc{$k$-Tournament-Inversion} is:

\begin{center}
\tcbox[
colback=white, colframe=black]{
\begin{tabular}{l}
\textsc{\underline{Input}:} a tournament $T$. \\
\textsc{\underline{Problem}:} $\inv(T)\leq k$?
\end{tabular}
}
\end{center}

\noindent One way of analysing the complexity of this problem is to use $k$-inversion-critical tournaments: we say that a tournament $T$ is a \emph{$k$-inversion-critical tournament} if $\inv(T)=k$ but $\inv(T-\{v\})<k$ for all $v\in V(T)$, and denote by $\mathcal{IC}_k$ the set of $k$-inversion-critical tournaments. It is not difficult to see that a tournament has inversion number at most $k$ if and only if it contains no element of $\mathcal{IC}_{k+1} \cup\mathcal{IC}_{k+2}$ as a subtournament. Indeed, for any digraph $D$ and vertex $v \in V(D)$ with out-neighbourhood $A \subseteq V(D)$, adding $A$ and $A\cup\{v\}$ to a decycling family of $D-\{v\}$ gives a decycling family of $D$. We deduce that $\inv(D) \leq \inv(D-\{v\})+2$. Hence, if $\inv(T)>k$, then by arbitrarily deleting vertices from $T$ one by one, we can obtain a subtournament $T'$ of $T$ with $\inv(T')\in \{k+1,k+2\}$. This $T'$ contains a member of $\mathcal{IC}_{k+1} \cup\mathcal{IC}_{k+2}$ as a subtournament.

Belkhechine, Bouaziz, Boudabbous, and Pouzet~\cite{BELKHECHINE2010703} showed that $\mathcal{IC}_k$ is finite for all $k\in\N$. Writing $m_k$ for the maximum number of vertices of an element of $\mathcal{IC}_k$, it follows that \textsc{$k$-Tournament-Inversion} can be solved in time $O(\abs{V(T)}^{\max(m_{k+1},m_{k+2})})$. Thus, in particular, \textsc{$k$-Tournament-Inversion} can be solved in polynomial time for any fixed $k$. Plainly $\mathcal{IC}_1=\{\vv{C_3}\}$, so $m_1=3$, and $\mathcal{IC}_2$ was explicitly described in~\cite{pouzet}, giving $m_2=6$. However, no upper bound on $m_k$ is known for $k\geq 3$, so for no $k\geq 1$ does the above give a concrete polynomial bound on the complexity of \textsc{$k$-Tournament-Inversion}. Note also that this approach does not identify a decycling family of size $k$ given a $k$-invertible tournament, it can only confirm the existence of one. 

Bang-Jensen, da Silva, and Havet~\cite{bangjensen2022inversion} used an alternative approach to show that \textsc{$1$-Tournament-Inversion} can be solved in time $O(\abs{V(T)}^3)$ while \textsc{$2$-Tournament-Inversion} can be solved in time $O(\abs{V(T)}^6)$. The idea behind their algorithm for \textsc{$1$-Tournament-Inversion} is to check whether the tournament contains a vertex which can be made into a source, and for \textsc{$2$-Tournament-Inversion} they check whether it contains a pair of vertices which can be made into a source and a sink respectively.
They went on to ask for the least real numbers $r_k$ such that \textsc{$k$-Tournament-Inversion} can be solved in time $O(\abs{V(T)}^{r_k})$. We answer this question by showing that, perhaps surprisingly, for each fixed $k\in\N$ there is an algorithm solving \textsc{$k$-Tournament-Inversion} in time $O(\abs{V(T)}^2)$. In the language of complexity theory, this means that the likely NP-hard problem of determining whether $\inv(T)\leq k$ for inputs $k$ and $T$ (see \cref{conj:tnmnt_NPcomplete}) is fixed-parameter tractable when parameterised by $k$.\footnote{See \cite{downey2} for the definition of fixed-parameter tractability and an exposition of the surrounding theory.}

\begin{theorem}\label{thm:tournament_fpt}
For fixed $k\in\mathbb{N}$,  \textsc{$k$-Tournament-Inversion} can be solved for $n$-vertex tournaments in time $O(n^2)$. Moreover, if the input tournament is $k$-invertible, then our algorithm finds a decycling family of size at most $k$.
\end{theorem}

Note that the exponent of $n$ in this running time is optimal, since any algorithm solving \textsc{$k$-Tournament-Inversion} needs to inspect the orientation of every edge in the input tournament.
However, the implied constant in the running time of our algorithm is doubly exponential in~$k$, so it is unlikely to be of practical use for large~$k$. 

\subsection{Relation to other parameters}\label{subsec:intro_tau}

Bang-Jensen, da Silva, and Havet~\cite{bangjensen2022inversion} also considered the relationship between the inversion number and other digraph parameters. 
Two well studied parameters of particular interest are the cycle transversal number and the cycle edge-transversal number, defined as follows. A \emph{cycle transversal} (or \emph{feedback vertex set}) in a digraph $D$ is a set of vertices of $D$ whose removal from $D$ leaves an acyclic digraph and the \emph{cycle transversal number} of $D$, denoted $\tau(D)$, is the minimum size of a cycle transversal in $D$. Analogously, a \emph{cycle edge-transversal} (or \emph{feedback arc set}) in $D$ is a set of edges of $D$ whose removal leaves an acyclic digraph and the \emph{cycle edge-transversal number} of $D$, $\tau'(D)$, is the minimum size of a cycle edge-transversal in $D$. 
Note that the inequality $\tau(D)\leq 2\tau'(D)$ always holds, since the endpoints of the edges in a cycle edge-transversal of $D$ form a cycle transversal of $D$.

Bang-Jensen, da Silva, and Havet~\cite{bangjensen2022inversion} made the following observations concerning the relationships between $\inv(D)$, $\tau(D)$, and $\tau'(D)$ for a digraph $D$. Firstly, we have $\inv(D)\leq \tau'(D)$. This follows from the fact that if $F\subseteq E(D)$ is a cycle edge-transversal of $D$, then since $(V(D),E(D)\setminus F)$ is acyclic, there is a labelling $v_1,\dots,v_n$ of $V(D)$ such that $\vv{v_jv_i}\not\in E(D)\setminus F$ if $i<j$. Applying the family of inversions $(\{v_i,v_j\}\colon i<j,\vv{v_jv_i}\in F)$ transforms $D$ into an acyclic digraph and hence $\inv(D)\leq \tau'(D)$ as claimed. They also observed that this inequality is tight for all values of $\tau'(D)$ as exhibited by $[\vv{C_3}]_k$, which clearly has cycle edge-transversal number $k$, and as mentioned above was shown in~\cite{pouzet} to have inversion number $k$.

Turning to $\tau(D)$, the inequality $\inv(D)\leq 2\tau(D)$ was obtained in~\cite{bangjensen2022inversion} as follows. After observing that $\tau(D)=0$ implies $\inv(D)=0$, we may assume that $\tau(D)\geq 1$. Let $S\subseteq V(D)$ be a cycle transversal in $D$ of size $\tau(D)$ and pick $v\in S$. Then observe that $D-\{v\}$ has cycle transversal number $\tau(D)-1$, with $S\setminus\{v\}$ a cycle transversal. Moreover, as noted in Section 1.2 we have $\inv(D) \leq \inv(D-\{v\})+2$, from which it follows by induction that $\inv(D)\leq 2\tau(D)$.

Bang-Jensen, da Silva, and Havet conjectured that this inequality is tight for all values of $\tau(D)$. Indeed, they considered the graph $V_5$ obtained by adding a vertex $v$ and edges $\vv{v1},\vv{2v},\vv{v3},\vv{4v}$ to the (transitive) tournament on vertex set $\{1,2,3,4\}$ with edges $\vv{ij}$ for $i<j$, which can easily be seen to have $\tau(V_5)=1$ and $\inv(V_5)=2$. They noted that if the dijoin conjecture holds, then $\tau([V_5]_k)=k$ and $\inv([V_5]_k)=2k$ for all $k$ (in fact, since $V_5$ is strongly connected, the case $k=2$ follows from the special cases for which they proved the dijoin conjecture). We construct digraphs with a similar character to $V_5$ which confirm their conjecture.

\begin{theorem}\label{thm:tau}
For all $k\in\N$ there exists an oriented graph $D$ with $\inv(D)=2\tau(D)=2k$.
\end{theorem}

\subsection{The extremal problem}\label{subsec:intro_inv(n)}
Finally, we consider $\inv(n)$, defined for each $n\in\N$ as the maximum inversion number of an oriented graph (or, equivalently, a tournament) on $n$ vertices. Belkhechine, Bouaziz, Boudabbous, and Pouzet~\cite{BELKHECHINE2010703} were the first to study this parameter, obtaining bounds of the form\footnote{All logarithms in this paper are taken base 2.}
\[\frac{n}{2}-\log(n)+O(1)\leq \inv(n)\leq n + O(1).\]
Their lower and upper bounds follow from counting and inductive arguments respectively (see \cref{sec:inv(n)} for details), and they conjectured that $\inv(n)\geq \floor{\frac{n-1}{2}}$ for all $n$. Bounds of the form above previously remained the best known, with Bang-Jensen, da Silva, and Havet~\cite{bangjensen2022inversion} noting that the $O(1)$ term in the upper bound can be improved very slightly. 

Using a random construction, we show that $\inv(n)=(1+o(1))n$.

\begin{theorem}\label{thm:inv_lower}
For sufficiently large $n$ we have
\[\inv(n) \geq n- \sqrt{2 n \log(n)}.\]
Moreover, a uniformly random labelled $n$-vertex tournament has at least this inversion number with probability tending to $1$.
\end{theorem}

In \cref{sec:inv(n)} we also show that $\inv(n) \leq n-\log(n+1)$.

\subsection{Outline of the paper}\label{subsec:outline}
The remainder of the paper is organised as follows. 
In \cref{sec:prelims} we introduce some further notation, definitions, and preliminary observations which will be useful in the rest of the paper. 
In the very short \cref{sec:dijoin} we prove \cref{thm:dijoin_ce}, constructing a counterexample to the dijoin conjecture. 
Our results on the inversion number of $k$-joins, \cref{thm:2ra2}, \cref{thm:inv_kjoin}, and \cref{thm:kjoin_char}, are proved in \cref{sec:kjoins}, along with \cref{thm:NPcomplete}.
\Cref{sec:tournament_fpt} concerns the complexity of \textsc{$k$-Tournament-Inversion} and contains the proof of \cref{thm:tournament_fpt}.
We give the proof of \cref{thm:tau} in \cref{sec:tau}. In
\cref{sec:inv(n)} we discuss the existing bounds on $\inv(n)$ before proving \cref{thm:inv_lower} and giving an improved upper bound. 
Finally, in \cref{sec:conclusion} we restate some conjectures and questions from previous papers which remain open and pose some new ones of our own.

\noindent\textbf{Note added.} Almost simultaneously with the initial release of this paper, Aubian, Havet, H\"{o}rsch, Klingelhoefer, Nisse, Rambaud, and Vermande announced independent work~\cite{aubian} on some of the problems we address here. Specifically, they prove a stronger version of \cref{thm:dijoin_ce} (in fact, they prove that a strong version of our \cref{conj:dijoin_false} holds provided at least one of $\ell$ and $r$ is odd and at least 3) and they show upper and lower bounds on $\inv(n)$ of forms similar to those we give in \cref{subsec:intro_inv(n)}. 

\section{Notation and preliminaries}\label{sec:prelims}

In this section we detail some of the definitions, observations, and notation to be used in the rest of the paper. As noted above, all digraphs will be oriented graphs, that is, loopless directed graphs with at most one edge between each pair of vertices. An \emph{acyclic digraph} is a digraph with no directed cycles. In the case where the digraph is a tournament, we use the term \emph{transitive} instead of acyclic. Note that for each $n\in\N$ there is a unique unlabelled transitive tournament on $n$ vertices. To a transitive tournament $T$ we associate the total order $<$ on $V(T)$ where $u<v$ for all $u,v\in V(T)$ such that $\vv{uv}\in E(T)$. We write $[n]$ for the set $\{1,2,\dots,n\}$. For a digraph $D$ and a set $S\subseteq V(D)$ we write $D-S$ for the digraph produced by deleting the vertices in $S$ from $D$. We now give the following key definitions.

\begin{definition}\label{def:charvec_atom}
Recall that for a digraph $D$, sets $X_1,\dots,X_k\subseteq V(D)$, and a vertex $v\in V(D)$, the \emph{characteristic vector} of $v$ in $X_1,\dots,X_k$ is $(I_{\{v\in X_i\}}:i\in[k])\in\F_2^k$, where $I_{\{v\in X_i\}}$ is the indicator function of the event $v\in X_i$. Define an equivalence relation $\sim$ on $V(D)$ by setting $u\sim v$ if $u$ and $v$ have the same characteristic vector in $X_1,\dots,X_k$. The \emph{atoms} of $X_1,\dots,X_k$ in $D$ are the equivalence classes of this relation.
\end{definition}

Note that, equivalently, the atoms of $X_1,\dots,X_k$ in $D$ are the atoms of the set algebra on $V(D)$ generated by $X_1,\dots,X_k$, and that there are at most $2^k$ atoms for given $D$ and $X_1,\dots,X_k$. The next observation will be useful throughout the paper.  

\begin{observation}\label{obs:atoms}
Let $D$ be a digraph and suppose that $u,v\in V(D)$ are joined by an edge in $D$. Let $X_1,\dots,X_k\subseteq V(D)$. Write $\mbf{u},\mbf{v}\in\F_2^k$ for the characteristic vectors of $u$ and $v$ in $X_1,\dots,X_k$ respectively. Then the edge between $u$ and $v$ undergoes a net change in orientation when $X_1,\dots,X_k$ are inverted in $D$ if and only if $\mbf{u}\cdot\mbf{v}=1$.
\end{observation}

This follows from the fact that $\mbf{u}\cdot\mbf{v}$ is the parity of the number of $X_1,\dots,X_k$ which contain both $u$ and $v$. An obvious implication of \cref{obs:atoms} is that given $D$ and $X_1,\dots,X_k$, for every pair of (not necessarily distinct) atoms $A$ and $B$, either all edges $\{ab \colon a\in A, b\in B\}$ undergo a net orientation change when $X_1,\dots,X_k$ are inverted, or none of them do. In particular, for every vertex $v$ and atom $A$, either all edges $\{va \colon a\in A\}$ change orientation or none of them do.

Finally, we note some simple observations which will be used freely in what follows.
\begin{enumerate}[nolistsep, label=(\roman*)]
    \item If $D'$ is a subdigraph of an oriented graph $D$, then $\inv(D')\leq \inv(D)$.
    \item For every oriented graph $D$ and every non-negative integer $k\leq \inv(D)$, there exists a spanning subdigraph of $D$ with inversion number $k$.
    \item If $X_1,\dots,X_k$ is a decycling family of an oriented graph $D$, then $D$ can be extended to a tournament $T$ for which $X_1,\dots,X_k$ is still a decycling family. In particular $\inv(T)=\inv(D)$.
\end{enumerate}
For (ii), delete edges of $D$ one by one, noting that the inversion number drops by at most 1 at each step. For (iii), invert the decycling family in $D$, extend the resulting acyclic digraph to a transitive tournament, then invert the decycling family again.

\section{A counterexample to the dijoin conjecture}\label{sec:dijoin}

In this short section we give a counterexample to the dijoin conjecture of Bang-Jensen, da Silva, and Havet~\cite{bangjensen2022inversion}, that is, the conjecture that $\inv(L \ra R)=\inv(L)+\inv(R)$ for all oriented graphs $L$ and $R$. As noted in the introduction, this conjecture is equivalent to its restriction to tournaments. Indeed, suppose that $L$ and $R$ are digraphs with $\inv(L\ra R)<\inv(L)+\inv(R)$. Extend $L\ra R$ to a tournament of the same inversion number and observe that this tournament is $L'\ra R'$ for some tournaments $L'$ and $R'$ extending $L$ and $R$ respectively. These clearly satisfy $\inv(L')\geq \inv(L)$ and $\inv(R')\geq \inv(R)$, so we have tournaments $L'$ and $R'$ with $\inv(L'\ra R')<\inv(L)+\inv(R)\leq \inv(L')+\inv(R')$.

\begin{proof}[Proof of \cref{thm:dijoin_ce}]
Let $L$ be a copy of $\vv{C_3}$. Suppose that $R$ is a tournament with $\inv(R)=3$ for which there exist disjoint $A,B,C \subseteq V(R)$ such that $A\cup B$, $A \cup C$ and $B\cup C$ form a decycling family of $R$. Then for distinct vertices $u,v\in V(L)$ the sets $A\cup B \cup \{u,v\}$, $A \cup C \cup \{u,v\}$ and $B\cup C \cup \{u,v\}$ form a decycling family of $L\ra R$, demonstrating that \[\inv(L\ra R)=3<4=\inv(L)+\inv(R).\]

One way to construct such an $R$ is as follows: let $R$ be the tournament with vertex set $[9]$, let $A=\{1,3\}$, $B=\{4,6\}$, and $C=\{7,9\}$, and let the edge $ij$ be directed backwards (that is, from $j$ to $i$ when $i<j$) if and only if $i$ and $j$ are both in $A\cup B\cup C$, but not both in $A$, $B$, or $C$. It is clear that inverting $A\cup B$, $A \cup C$ and $B\cup C$ transforms $R$ into a transitive tournament, and a computer check shows that $\inv(R)=3$, as required.
\end{proof}

\section{Decycling families of \texorpdfstring{$k$-joins}{k-joins}}\label{sec:kjoins}

In this section we prove \cref{thm:kjoin_char}, which characterises the decycling families of size $k$ of $k$-joins of digraphs each with inversion number $1$. We will then deduce \cref{thm:inv_kjoin} from this characterisation, and use \cref{thm:inv_kjoin} to obtain \cref{thm:2ra2} and \cref{thm:NPcomplete}.
The bulk of the work in our proof of \cref{thm:kjoin_char} is put towards proving \cref{lem:C3_kjoin}, which deals with the case $\hD=[\vv{C_3}]_k$.

\begin{lemma}\label{lem:C3_kjoin}
Let $k\in\N$, let $\hD=[\vv{C_3}]_k$, and let $X_1,\dots,X_k\subseteq V(\hD)$ be a decycling family of $\hD$. Then there exist orthonormal vectors $\mbf{u}_1,\dots,\mbf{u}_k\in\F_2^k$ such that in the $i$th factor of $[\vv{C_3}]_k$, one vertex has characteristic vector $\mbf{0}$ and the other two have characteristic vector~$\mbf{u}_i$.
\end{lemma}

We will use the setup that Pouzet, Kaddour, and Thatte~\cite{pouzet} introduced in their proof that $\inv(\hD)=k$. The first part of our argument is essentially a reformulation of theirs, but we include it for completeness and to build intuition.

\begin{proof}[Proof of \cref{lem:C3_kjoin}]
Let $T$ be the transitive tournament obtained by inverting the sets $X_1,\dots,X_k$ in $\hD$, and let $<$ be the total order on $V(\hD)$ associated to~$T$. Note that for all $i$, after inverting $X_1, \dots, X_k$ the $i$th factor has one vertex that has out-edges to the other two vertices in the factor and exactly one of these edges has undergone a net reversal. Thus we can label the vertices in the $i$th factor as $u_i,v_i,w_i$ where $\vv{u_iv_iw_i}$ is a directed 3-cycle in $\hD$, and the edge between $u_i$ and $w_i$ undergoes a net reversal under $X_1,\dots,X_k$ while the edge between $u_i$ and $v_i$ does not. In particular, we will use throughout that $u_i<v_i,w_i$ and that, by \cref{obs:atoms}, $\mbf{u}_i\cdot\mbf{v}_i=0$ and $\mbf{u}_i\cdot\mbf{w}_i=1$ where $\mbf{u}_i,\mbf{v}_i,\mbf{w}_i\in\F_2^k$ are the respective characteristic vectors of $u_i,v_i,w_i$ in $X_1,\dots,X_k$.  We have the following claim, originally proved in~\cite{pouzet}.

\begin{claim}[\cite{pouzet}]\label{claim:linindep}
The vectors $\mbf{u}_1,\dots,\mbf{u}_k\in\F_2^k$ are linearly independent.
\end{claim}
\begin{proof}
The statement is equivalent to the claim that for all non-empty $I\subseteq[k]$ we have $\sum_{i\in I}\mbf{u}_i\neq \mbf{0}$. Fix such an $I$ and note that it is sufficient to show that there exists some $\mbf{x}\in\F_2^k$ such that $(\sum_{i\in I}\mbf{u}_i)\cdot\mbf{x}\neq 0$. Let $m\in I$ be such that $u_i<u_m$ for all $i\in I\setminus\{m\}$. Note that $u_m<v_m,w_m$, so by the transitivity of $T$ we have $u_i<v_m,w_m$ for all $i\in I$. It is straightforward to deduce from this that for all $i\in I\setminus\{m\}$, the orientations of the edges $u_iv_m$ and $u_iw_m$ are either both unchanged after $X_1,\dots,X_k$ are inverted, or both reversed. By \cref{obs:atoms}, in other words we have $\mbf{u}_i\cdot\mbf{v}_m=\mbf{u}_i\cdot\mbf{w}_m$ for all $i\in I\setminus\{m\}$. On the other hand we have $\mbf{u}_m\cdot\mbf{v}_m=0$ while $\mbf{u}_m\cdot\mbf{w}_m=1$, so it follows by linearity of the dot product that $(\sum_{i\in I}\mbf{u}_i)\cdot\mbf{v}_m\neq (\sum_{i\in I}\mbf{u}_i)\cdot\mbf{w}_m$. One of these two dot products is thus non-zero, and we deduce that $\sum_{i\in I}\mbf{u}_i\neq \mbf{0}$, as required.
\end{proof}

We now build on \cref{claim:linindep} as follows.

\begin{claim}\label{claim:v=0_u=w}
Let $\ell\in[k]$ and suppose that the vectors $\mbf{u}_i,\mbf{v}_i,\mbf{w}_i$ for $\ell\leq i\leq k$ all lie in a subspace $V$ of $\F_2^k$ of dimension $k-\ell+1$. Then $\mbf{u}_\ell,\dots,\mbf{u}_k$ are orthonormal, and for all $\ell\leq i\leq k$ we have $\mbf{u}_i=\mbf{w}_i$ and $\mbf{v}_i=\mbf{0}$.
\end{claim}
\begin{proof}
We will prove the claim by reverse induction on $\ell$. In the $\ell=k$ case the claim follows easily from the fact that $\mbf{u}_k\cdot\mbf{w}_k=1$ while $\mbf{u}_k\cdot\mbf{v}_k=0$. Thus, let $\ell\leq k-1$ and write $[\ell,k]$ for $\{\ell,\ell+1,\dots,k\}$. Let $z$ be the $<$-minimal vertex among $v_\ell,\dots,v_k,w_\ell,\dots,w_k$. Write $\mbf{z}\in V\subseteq\F_2^k$ for the characteristic vector of $z$ in $X_1,\dots,X_k$ and let $t\in[\ell,k]$ be such that $z\in\{v_t,w_t\}$. By \cref{claim:linindep}, the vectors $\mbf{u}_\ell,\dots,\mbf{u}_k$ form a basis of $V$ so there exists $I\subseteq[\ell,k]$ such that $\mbf{z}+\sum_{i\in I}\mbf{u}_i=\mbf{0}$. 

First suppose that $I\not\in\{\emptyset,\{t\}\}$ and let $m\in I$ be such that $u_i<u_m$ for all $i\in I\setminus\{m\}$. If $m\neq t$, then we have $z<v_m,w_m$, so $\mbf{z}\cdot\mbf{v}_m=\mbf{z}\cdot\mbf{w}_m$ by \cref{obs:atoms}. As in the proof of \cref{claim:linindep}, we have $\mbf{u}_i\cdot\mbf{v}_m=\mbf{u}_i\cdot\mbf{w}_m$ for all $i\in I\setminus\{m\}$, but $\mbf{u}_m\cdot\mbf{v}_m\neq\mbf{u}_m\cdot\mbf{w}_m$, so $(\mbf{z}+\sum_{i\in I}\mbf{u}_i)\cdot\mbf{v}_m\neq (\mbf{z}+\sum_{i\in I}\mbf{u}_i)\cdot\mbf{w}_m$, and hence $\mbf{z}+\sum_{i\in I}\mbf{u}_i\neq \mbf{0}$. If $m=t$, then let $j\in I\setminus\{t\}$ and note that $z<v_{j},w_{j}$ by the minimality of $z$. Consequently, $\mbf{z}\cdot\mbf{v}_j=\mbf{z}\cdot\mbf{w}_j$. Moreover, since $u_m=u_t<z$, we have $u_m<v_{j},w_{j}$. From this it follows that $u_i<v_{j},w_{j}$ for all $i\in I$. Thus, $\mbf{u}_i\cdot\mbf{v}_{j}=\mbf{u}_i\cdot\mbf{w}_{j}$ for all $i\in I\setminus\{j\}$, while $\mbf{u}_{j}\cdot\mbf{v}_{j}\neq\mbf{u}_{j}\cdot\mbf{w}_{j}$. Hence, similarly to above, we have $(\mbf{z}+\sum_{i\in I}\mbf{u}_i)\cdot\mbf{v}_j\neq (\mbf{z}+\sum_{i\in I}\mbf{u}_i)\cdot\mbf{w}_j$, so $\mbf{z}+\sum_{i\in I}\mbf{u}_i\neq \mbf{0}$.

The remaining cases are $I=\emptyset$ and $I=\{t\}$, so we have $\mbf{z}\in\{\mbf{0},\mbf{u}_t\}$. Suppose that $\mbf{z}=\mbf{u}_t$. If $z=v_t$, then we have $\mbf{v_t}=\mbf{z}=\mbf{u_t}$, so $\mbf{v_t} \cdot \mbf{w_t}=\mbf{u_t} \cdot \mbf{w_t}=1$, i.e.\ the edge between $v_t$ and $w_t$ undergoes a net reversal under $X_1,\dots,X_k$. This would imply that $w_t<v_t=z$, which contradicts the minimality of $z$. Similarly, if $z=w_t$, then since the edge between $u_t$ and $v_t$ is not inverted, neither is the edge between $w_t$ and $v_t$, so $v_t<w_t=z$, another contradiction. Therefore $\mbf{z}=\mbf{0}$. This means no edges incident to $z$ are reversed when $X_1,\dots,X_k$ are inverted so by the minimality of $z$ we have $z=v_\ell$. 

We have shown that $\mbf{v}_\ell=\mbf{0}$, so the only vertex among the $u_i,v_i,w_i$ with $i\geq \ell$ which precedes $v_\ell$ in $<$ is $u_\ell$. It follows that $u_\ell$ is the least element among the $u_i,v_i,w_i$ with $i\geq \ell$. Hence, by \cref{obs:atoms} we have $\mbf{u}_\ell\cdot\mbf{u}_i=\mbf{u}_\ell\cdot\mbf{v}_i=\mbf{u}_\ell\cdot\mbf{w}_i=0$ for all $i\geq \ell+1$, so if $V'$ is the subspace of $V$ spanned by the $\mbf{u}_i,\mbf{v}_i,\mbf{w}_i$ with $i\geq \ell+1$, then $\mbf{u}_\ell\cdot\mbf{x}=0$ for all $\mbf{x}\in V'$. We have $\mbf{u}_\ell\cdot\mbf{w}_\ell=1$, so $V'$ is a proper subspace of $V$, but $\mbf{u}_{\ell+1},\dots,\mbf{u}_k\in V'$ are linearly independent, so we deduce that $V'$ has dimension $k-\ell$. Therefore by the induction hypothesis $\mbf{u}_{\ell+1},\dots,\mbf{u}_k$ are orthonormal, and we have $\mbf{u}_i=\mbf{w}_i$ and $\mbf{v}_i=\mbf{0}$ for $i\geq \ell+1$.

To complete the induction step it remains to show that $\mbf{u}_\ell=\mbf{w}_\ell$ and $\mbf{u}_\ell\cdot\mbf{u}_\ell=1$. The latter follows from the fact that $\mbf{u}_\ell,\dots,\mbf{u}_k$ is a basis for $V$ with $\mbf{u}_\ell\cdot\mbf{u}_i=0$ for all $i\geq \ell+1$, but $\mbf{w}_\ell\in V$ has $\mbf{u}_\ell\cdot\mbf{w}_\ell=1$. For the former, note that $\mbf{w}_\ell=\sum_{i\in I}\mbf{u}_i$ for some $I\subseteq [\ell,k]$ and by the established properties of the $\mbf{u}_i$ this set $I$ contains exactly those $i$ for which $\mbf{u}_i\cdot\mbf{w}_\ell=1$. Thus, we certainly have $\ell\in I$. Suppose that $\mbf{u}_i\cdot\mbf{w}_\ell=1$ for some $i\geq \ell+1$. Since $\mbf{w}_i=\mbf{u}_i$ and $\mbf{v}_i=\mbf{0}$, by \cref{obs:atoms} we find that the cycle $\vv{w_\ell v_iw_i}$ appears in $T$, which is a contradiction. Hence $I=\{\ell\}$ and $\mbf{u}_\ell=\mbf{w}_\ell$, as required.
\end{proof}
The lemma now follows from the $\ell=1$ case of \cref{claim:v=0_u=w}.
\end{proof}

We will now deduce \cref{thm:kjoin_char} from the lemma. In the proof, we will use the easy fact that every family of orthonormal vectors in $\F_2^k$ is linearly independent.

\begin{proof}[Proof of \cref{thm:kjoin_char}]
The sufficiency of the given conditions for $X_1,\dots,X_k$ to be a decycling family of $\hD$ is straightforward to verify using \cref{obs:atoms}. This observation also allows the `in particular' part of the theorem statement to be easily deduced from the preceding part. It remains to prove that the given conditions are necessary.

Given a decycling family $X_1,\dots,X_k$ of $\hD$, extend $\hD$ to a tournament $T$ for which $X_1, \dots, X_k$ is still a decycling family. For each $i$, let $T_i$ be the subtournament of $T$ induced on the vertex set of the $i$th factor of $\hD$. Since $D_i$ contains a directed cycle, so does $T_i$, and hence the latter contains a copy of $\vv{C_3}$. We can thus find a copy of $[\vv{C_3}]_k$ in $T$ whose $i$th factor is contained in $T_i$. It follows by \cref{lem:C3_kjoin} that there are orthonormal vectors $\mbf{u}_1, \dots, \mbf{u}_k\in \F_2^k$ and for each $i$ a triangle $\vv{u_iv_iw_i}$ in $T_i$ such that $u_i$ and $w_i$ have characteristic vector $\mbf{u}_i$ and $v_i$ has characteristic vector $\mbf{0}$ in $X_1, \dots, X_k$.

We next show that for all $i$, all vertices in $T_i$ have characteristic vector either $\mbf{u}_i$ or $\mbf{0}$ in $X_1,\dots,X_k$. Let $z \in V(T_i)$ and let $\mbf{z}$ be its characteristic vector. Since $\mbf{u}_1, \dots, \mbf{u}_k$ form a basis of $\F_2^k$, there exists $J\subseteq [k]$ such that $\mbf{z}=\sum_{j\in J}\mbf{u}_j$. If there exists $\ell \in J\setminus \{i\}$, then $\mbf{z} \cdot \mbf{u}_\ell=\mbf{u}_\ell \cdot \mbf{u}_\ell=1$ and hence the directions of the edge between $z$ and $u_\ell$ and the edge between $z$ and $w_\ell$ are reversed under $X_1, \dots, X_k$. If $\ell<i$, then the cycle $\vv{u_\ell v_\ell z}$ appears in $T$ and if $i<\ell$, then the cycle $\vv{z v_\ell w_\ell}$ appears in $T$. We have a contradiction in both cases, so $J=\emptyset$ or $J=\{i\}$ as desired. 

We have shown that all vertices in the $i$th factor of $\hD$ have characteristic vector either $\mbf{u}_i$ or $\mbf{0}$ in $X_1,\dots,X_k$. The effect on this copy of $D_i$ of inverting these sets in $\hD$ is therefore the same as inverting the set  of vertices with characteristic vector $\mbf{u}_i$, which we call $V_i$. The latter is therefore a decycling set for the $i$th factor of $\hD$. This completes the proof of the theorem.
\end{proof}

\Cref{thm:inv_kjoin} now follows easily.
\begin{proof}[Proof of \cref{thm:inv_kjoin}]
It is clear that the left-hand side of equation~\eqref{eq:thm_nested} is at most the right-hand side. For the reverse inequality, let $\hD=[D_1,\dots,D_k]$ and note that we may assume that none of the $D_i$ have inversion number $0$. Indeed, if $\inv(D_i)=0$ for some $i\geq 2$, then view $\hD$ as the $(k-1)$-join $[D_1,\dots,D_{i-2},D_{i-1}\ra D_i, D_{i+1},\dots,D_k]$ and, since $\inv(D_{i-1}\ra D_i)=\inv(D_{i-1})$, the result follows by induction on $k$. The case where $i=1$ can be handled similarly.

Thus, consider the case where $\inv(D_i)=1$ for all $i$ and suppose for a contradiction that $X_1,\dots,X_k$ is a decycling family of $\hD$ with $X_k=\emptyset$. By \cref{thm:kjoin_char} there exist $k$ orthonormal, and hence linearly independent, vectors in $\F_2^k$ each of which occurs as the characteristic vector of some vertex of $\hD$ in $X_1,\dots,X_k$. This contradicts the fact that all such vectors have a 0 in their final coordinate. Hence, in this case, $\inv(\hD)=k$.

It remains to check the case where $\inv(D_j)=2$ for some $j$ and $\inv(D_i)=1$ for all $i\neq j$. Start by letting $D'_j$ be a spanning subdigraph of the $j$th factor of $\hD$
with $\inv(D'_j)=1$, then define $\hD'$ to be the digraph obtained by replacing the $j$th factor of $\hD$ by $D'_j$. Assume for a contradiction that $X_1,\dots,X_k$ is a decycling family of $\hD$, in which case it is also a decycling family of $\hD'$. \cref{thm:kjoin_char} thus yields a vector $\mbf{u}_j \in \F_2^k$ with $\mbf{u}_j \cdot \mbf{u}_j = 1$ such that all the vertices in the $j$th factor of $\hD'$ (and hence also the $j$th factor of $\hD$) have characteristic vector either $\mbf{0}$ or $\mbf{u}_j$ in $X_1,\dots,X_k$. Inverting $X_1,\dots,X_k$ in~$\hD$ therefore has the same effect on its $j$th factor as inverting the set of vertices with characteristic vector $\mbf{u}_j$. It follows that this set of vertices is a decycling set for $D_j$, contradicting $\inv(D_j)=2$. 
\end{proof}

As mentioned in the introduction, it follows from \cref{thm:inv_kjoin} that for any digraph~$D$ we have $\inv(D)=1$ if and only if $\inv([D]_k)=k$, which in turn implies \cref{thm:NPcomplete} (which states that \textsc{$k$-Inversion} is NP-complete for all $k\in\N$). Indeed, \cref{thm:inv_kjoin} directly gives $\inv([D]_k)=k$ in the case $\inv(D)=1$, and if $\inv(D)=0$ then clearly $\inv([D]_k)=0$. If $\inv(D)>1$, then there are subdigraphs $D'$ and $D''$ of $D$ with $\inv(D')=1$ and $\inv(D'')=2$. The $k$-join $D''\ra[D']_{k-1}$, which has inversion number $k+1$ by \cref{thm:inv_kjoin}, is a subdigraph of $[D]_k$ and thus $\inv([D]_k)\geq k+1$ as required.

Finally, we deduce \cref{thm:2ra2} (which states that $\inv(L\ra R)=4$ for all digraphs $L$ and $R$ with inversion number $2$) from \cref{thm:inv_kjoin}. We will use the fact, shown in~\cite{bangjensen2022inversion}, that if $L$ and $R$ are strongly connected digraphs with $\inv(L),\inv(R)\geq 2$, then $\inv(L\ra R)\geq 4$.

\begin{proof}[Proof of \cref{thm:2ra2}]
Let $L$ and $R$ be digraphs with $\inv(L)=\inv(R)=2$. It is immediate that $\inv(L\ra R)\leq 4$, so it is sufficient to prove the lower bound. For this, extend $L\ra R$ to a tournament $T$ of the same inversion number and let the tournaments to which $L$ and $R$ are extended be $L'$ and $R'$ respectively. Note that $\inv(L'),\inv(R')\geq 2$ and $T$ is $L'\ra R'$.

Every tournament can be written as the $k$-join of its strongly connected components, so let $L'$ be $[L_1,\dots,L_{k_1}]$ and $R'$ be $[R_1,\dots,R_{k_2}]$ for some $k_1,k_2\in \N$ and strongly connected tournaments $L_1,\dots,L_{k_1},R_1,\dots,R_{k_2}$.
Since $\inv(L')\geq 2$, either there is some $L_i$ with $\inv(L_i)\geq 2$, or there are $i<j$ such that $\inv(L_i)=\inv(L_j)=1$. An analogous condition holds for $R'$. If there are $i$ and $j$ such that $\inv(L_i),\inv(R_j)\geq 2$, then since $T$ contains $L_i\ra R_j$, we have $\inv(T)\geq\inv(L_i\ra R_j)\geq 4$ by the above result of~\cite{bangjensen2022inversion}. Otherwise, either there exist $i<j$ such that $\inv(L_i)=\inv(L_j)=1$, in which case $\inv(T)\geq \inv([L_i,L_j,R])= 4$ by \cref{thm:inv_kjoin}, or there exist $i<j$ with $\inv(R_i)=\inv(R_j)=1$, in which case it follows similarly that $\inv(T)\geq 4$.
\end{proof}

\section{Complexity of \texorpdfstring{\textsc{$k$-Tournament-Inversion}}{k-Tournament-Inversion}}\label{sec:tournament_fpt}

In this section we prove \cref{thm:tournament_fpt} by constructing, for each fixed $k\in\N$, an algorithm solving \textsc{$k$-Tournament-Inversion} in time $O(\abs{V(T)}^2)$.
Our proof uses a technique known as iterative compression; see~\cite{downey2} for a description of this method and other applications of it.
The most involved part of our proof concerns the `compression step' of the algorithm. This step is handled by the following lemma, which roughly says that for constant $k$, given an $n$-vertex tournament $T_0$ and a decycling family of $T_0$ of constant size, in time linear in $n$ we can find a decycling family of $T_0$ of size $k$ if one exists. Throughout this section, we represent a total order $<$ on a finite set $S=\{s_1,\dots,s_m\}$ by the tuple $(s_1,\dots,s_m)$ where $s_1<\dots<s_m$.

\begin{lemma}\label{lem:tnmnt_inv}
Fix $k,s\in\mathbb{N}$. There is an algorithm which solves the following problem for $n$-vertex tournaments in time $O(n)$:

\begin{center}
\tcbox[
colback=white, colframe=black]{
\begin{tabular}{l}
\textsc{\underline{Inputs}:} \\  \tabitem a tournament $T_0$; \\
     \tabitem a decycling family $X_1, \dots, X_s$ of $T_0$ (transforming $T_0$ into $T$, say); \\
     \tabitem the order on $V(T_0)$ associated to $T$. \\
\textsc{\underline{Outputs}:} \\ \textsc{Either} \\ \tabitem that $T_0$ is not $k$-invertible; \\
 \textsc{Or} \\ \tabitem a decycling family $Y_1,\dots,Y_k$ of $T_0$ (transforming $T_0$ into $T'$, say); \\
     \tabitem the order on $V(T_0)$ associated to $T'$.
\end{tabular}
}
\end{center}
\end{lemma}

We now use iterative compression to prove \cref{thm:tournament_fpt} before returning to \cref{lem:tnmnt_inv}.

\begin{proof}[Proof of \cref{thm:tournament_fpt}]
Fix $k\geq 1$. We will induct on $n$ to define an algorithm solving the following problem for $n$-vertex tournaments in time $C_k\cdot n^2$ for some constant $C_k$:

\begin{center}
\tcbox[
colback=white, colframe=black]{
\begin{tabular}{l}
\textsc{\underline{Input}:} \\ \tabitem a tournament $T_0$. \\
\textsc{\underline{Ouputs}:} \\ \textsc{Either} \\
                \tabitem that $T_0$ is not $k$-invertible;\\
                \textsc{Or} \\
                \tabitem a decycling family $Y_1,\dots,Y_k$ of $T_0$ (transforming $T_0$ into $T$, say); \\
                \tabitem the order on $V(T_0)$ associated to $T$. \\
\end{tabular} }
\end{center}
In particular, this algorithm solves \textsc{$k$-Tournament-Inversion}.

Fix $n\geq 2$ and assume that we have defined such an algorithm for all smaller tournaments. Let $T_0$ be an $n$-vertex tournament and pick some $v\in V(T_0)$. Applying the induction hypothesis, in time $C_k\cdot (n-1)^2$ we either find that $T_0-\{v\}$ is not $k$-invertible or we obtain a decycling family $X_1,\dots,X_k$ of $T_0-\{v\}$ and the order on $V(T_0)\setminus\{v\}$ associated to the transitive tournament obtained by inverting these sets in $T_0$. In the former case, it follows that $T_0$ is also not $k$-invertible and we can output that fact. In the latter case, let $A$ be the out-neighbourhood of $v$ in $T_0$, and define $X_{k+1}=A\cup\{v\}$ and $X_{k+2}=A$. Then $X_1,\dots,X_{k+2}$ is a decycling family of $T_0$, and we can obtain the order associated to the resulting transitive tournament by adding $v$ to the previous order as the maximal element. By \cref{lem:tnmnt_inv} we can now, in linear time, either find that $T_0$ is not $k$-invertible or obtain a decycling family $Y_1,\dots,Y_k$ of $T_0$ of size $k$ and the order associated to the transitive tournament obtained by inverting these sets in $T_0$. As required, this algorithm runs in time $C_k\cdot(n-1)^2+O(n)$, which is at most $C_k\cdot n^2$ if $C_k$ is large enough.
\end{proof}

It is left to prove \cref{lem:tnmnt_inv}. To this end, we describe an algorithm which explores what happens if, starting from $T$, we invert $X_1, \dots, X_s$ and~$k$ further sets $Y_1, \dots, Y_k$ to obtain a tournament $T_Y$, where $Y=(Y_1, \dots, Y_k)$. Since $T_Y$ is the tournament obtained by inverting $Y_1,\dots,Y_k$ in $T_0$, these $k$ sets are a decycling family of $T_0$ if and only if $T_Y$ is transitive. If we were to examine each possibility individually there would be too many for this exploration process to be tractable. However, the fact that we are starting from a transitive tournament $T$ makes it possible to identify cycles in the final tournament $T_Y$ without fully specifying the sets $Y_1,\dots,Y_k$. This means there are far fewer cases to consider, indeed few enough that the exploration process is linear in $n$ for fixed $k$ and $s$.

\begin{proof}[Proof of \cref{lem:tnmnt_inv}]
Fix $k,s\in\N$ and let $T_0$, $X_1,\dots,X_s$, and $T$ be as in the statement of the lemma. Let $n=\abs{V(T_0)}$ and label the vertices of $T_0$ as $u_1,\dots,u_n$ in $T$-increasing order. With notation as above, we wish to investigate for which $Y$ the tournament $T_Y$ is transitive. For each $Y$ we write $\mbf{u}_i\in\F_2^{s+k}$ for the characteristic vector of $u_i$ in $X_1,\dots,X_s,Y_1,\dots,Y_k$ (suppressing the dependence on $Y$ in the notation) and then let $\mbf{u}=(\mbf{u}_1,\dots,\mbf{u}_n)$. There is a bijective correspondence between $Y$ and $\mbf{u}$ and it will be more convenient to work with the latter, so let $T_{\mbf{u}}=T_Y$ and write $\cU$ for the set of all possible $\mbf{u}$. Our first aim is to determine in linear time whether there exists $\mbf{u}\in \cU$ such that $T_{\mbf{u}}$ is transitive, and to identify such a $\mbf{u}$ if so.

The tournament $T_{\mbf{u}}$ is transitive exactly when it contains no cyclic triples. It is straightforward to use \cref{obs:atoms} to show that this is equivalent to the condition that there are no $a<b<c$ in $[n]$ such that $\mbf{u}_a\cdot\mbf{u}_b=\mbf{u}_b\cdot\mbf{u}_c$ but $\mbf{u}_a\cdot\mbf{u}_b\neq \mbf{u}_a\cdot\mbf{u}_c$. We describe the triple $(\mbf{u}_a,\mbf{u}_b,\mbf{u}_c)$ as \emph{bad} if this occurs. Thus, $T_{\mbf{u}}$ is transitive if and only if $B(\mbf{u})=\{(\mbf{u}_a,\mbf{u}_b,\mbf{u}_c)\colon a<b<c\}$ contains no bad triples, and $T_0$ is $k$-invertible if and only if $\cB=\{B(\mbf{u})\colon \mbf{u}\in\cU\}$ contains a set which is free of bad triples. Our algorithm will construct this set $\cB$ and check whether any of its elements are free of bad triples. If one of these sets \emph{is} free of bad triples, then we need to be able to output a corresponding decycling family of $T_0$, so for each $B\in \cB$ we will also record some $\mbf{u}\in \cU$ for which $B=B(\mbf{u})$.

We will now explain how the above can be achieved in linear time. First note that we may assume that $n\geq 4$. Let $\cU'$ be the set of all possible vectors $\mbf{u}'=(\mbf{u_1},\dots,\mbf{u}_{n-1})$ of characteristic vectors of $u_1,\dots,u_{n-1}$ in $X_1,\dots,X_s,Y_1,\dots,Y_k$. For $\mbf{u}'\in\cU'$, let $B'(\mbf{u}')=\{(\mbf{u}_a,\mbf{u}_b,\mbf{u}_c)\colon 1\leq a<b<c\leq n-1\}$ and let $\cB'=\{B'(\mbf{u}')\colon \mbf{u}'\in\cU'\}$. We may assume inductively that there is a constant $C$ depending only on $k$ and $s$ such that in time $C\cdot (n-1)$ we can construct $\cB'$ and associate to each $B'\in \cB'$ some $\mbf{u}'\in\cU'$ such that $B'=B'(\mbf{u}')$. For the induction step, we need to show that we can use this to obtain in time $C$ the set $\cB$ and for each $B\in \cB$ some $\mbf{u}\in\cU$ such that $B=B(\mbf{u})$.

The key observation is that there are only $2^{s+k}$ possible characteristic vectors for each of $u_1,\dots,u_n$, so the number of triples of characteristic vectors is at most $2^{3(s+k)}$ and the sizes of $\cB$ and $\cB'$ are at most $2^{2^{3(s+k)}}$. In particular, there are only boundedly many pairs $(B',\mbf{u}_n)$ where $B'\in \cB'$ and $\mbf{u}_n$ is a possible characteristic vector for $u_n$. For each such pair, we can construct in bounded time the set $S(B',\mbf{u}_n)$ consisting of all triples in $B'$, and all triples of the form $(\mbf{v}_i,\mbf{v}_j,\mbf{u}_n)$ for $(\mbf{v}_1,\mbf{v}_2,\mbf{v}_3)\in B'$ and $1\leq i<j\leq 3$. It is not hard to see that $\cB$ equals the set of all sets $S(B',\mbf{u}_n)$ and that each $S(B',\mbf{u}_n)$ can be associated with the $\mbf{u}\in\cU$ formed by appending $\mbf{u}_n$ to the $\mbf{u'}\in\cU'$ associated with $B'$. Indeed, given $B'$ and $\mbf{u}_n$ and defining $\mbf{u}$ as in the previous sentence, since $n \geq 4$, we have $S(B',\mbf{u}_n)=B(\mbf{u})$. For the other direction, given $\mbf{u}=(\mbf{u_1},\dots,\mbf{u}_n)\in\cU$ and letting $\mbf{u'}=(\mbf{u}_1,\dots,\mbf{u}_{n-1})$, we have $B(\mbf{u})=S(B'(\mbf{u'}), \mbf{u}_n)$.

We can construct this set in bounded time and then forget about all but one of the elements of $\cU$ associated to each $B\in \cB$. Thus, in constant time we have obtained $\cB$ and for each $B\in\cB$ some $\mbf{u}\in\cU$ such that $B=B(\mbf{u})$, and the induction continues.

Once we have constructed $\cB$ in linear time, since it has bounded size we can check whether any of its members is free of bad triples in bounded time. If not, then $T_0$ is not $k$-invertible. If so, then pick $B\in\cB$ with no bad triples and use the $\mbf{u}$ associated to it to construct a decycling family $Y_1,\dots,Y_k$ of $T_0$. Let $T'$ be the transitive tournament obtained by inverting these sets in~$T_0$.

It remains to show that we can obtain the order on $V(T_0)$ associated to $T'$ in linear time. Inverting the sets $X_1,\dots,X_s,Y_1,\dots,Y_k$ transforms $T$ into $T'$, and we have the characteristic vector of each vertex in these sets as well as the order on the vertices associated to $T$. We can therefore in linear time obtain the atoms of these $s+k$ inversions and for each atom $A$ the restriction to $A$ of the order associated to $T$. By reversing the order on each atom whenever the edges within it undergo a net reversal under the inversions, we obtain the order on that atom associated to $T'$. The $T'$-minimal vertex is now the minimal vertex of one of the atoms under their current orderings. There are at most $2^{s+k}$ atoms so we can identify the $T'$-minimal vertex in constant time. After deleting this vertex from its atom, the second smallest vertex according to $T'$ is one of the new minimal vertices of the atoms so can be found in constant time again. Continuing in this way we can obtain the full ordering in linear time, as required.
\end{proof}

Note that the implicit constant in the running time given by this proof is doubly exponential in $s+k$.

\section{Cycle transversals}\label{sec:tau}
In this section we will prove \cref{thm:tau}, constructing for each $k\in\N$ a digraph $D$ with $\tau(D)=k$ and $\inv(D)=2k$. We will use the so-called Eventown theorem, proved independently by Berlekamp~\cite{berlekamp_1969} and Graver~\cite{GRAVER1975111}.

\begin{theorem}[Eventown~\cite{berlekamp_1969},~\cite{GRAVER1975111}]
\label{thm:eventown}
Let $n\in\N$ and let $\mathcal{F}\subseteq\mathcal{P}([n])$ be a family of subsets of $[n]$ such that $\abs{F_1\cap F_2}$ is even for all $F_1,F_2\in\mathcal{F}$. Then $\abs{\mathcal{F}}\leq 2^{\floor{n/2}}$.
\end{theorem}

For a digraph $D$ and vertices $u,v,w\in V(D)$, we will say that $u$ and $v$ \emph{differ} on $w$ if either $\vv{uw},\vv{wv}\in E(D)$ or $\vv{vw},\vv{wu}\in E(D)$. We are now ready to prove the theorem.

\begin{proof}[Proof of \cref{thm:tau}]
Fix $k\in\N$ and let $n\in\N$ be large and divisible by $2^k$. We will define a digraph $D$ on vertex set $\{u_0,\dots,u_{k-1},v_0,\dots,v_{n-1}\}$ and then show that it satisfies the conditions of the theorem. Start by including all directed edges $\vv{v_iv_j}$ for $i<j$, so that the subdigraph of $D$ induced on $\{v_0,\dots,v_{n-1}\}$ is a transitive tournament. For $i\in \{0,\dots,k-1\}$ and $j\in\{0,\dots,n-1\}$, add the edge $\vv{u_iv_j}$ if in the binary expansion of~$j$, the digit in the $2^i$ place is a $0$, and add the edge $\vv{v_ju_i}$ otherwise. For ease of exposition we will not include any edges among the $u_i$ (though including any combination of such edges would still give a valid construction), so this completes the definition of $D$. As noted above, the removal of the vertices $u_0,\dots,u_{k-1}$ from $D$ leaves an acyclic digraph, so $\tau(D)\leq k$. 

It remains to show that $\inv(D)\geq 2k$, as then $\inv(D)=2 \tau(D)=2k$ follows from $\inv(D)\leq 2 \tau(D)$. Suppose for a contradiction that $X_1,\dots,X_{2k-1}\subseteq V(D)$ form a decycling family of $D$ and let $D'$ be the acyclic digraph obtained by inverting these sets in~$D$.
Consider the characteristic vectors of $v_0,\dots,v_{n-1}$ in $X_1,\dots,X_{2k-1}$, which we will denote by $\mbf{v}_0,\dots,\mbf{v}_{n-1}\in\F_2^{2k-1}$ respectively. Let $K=2^k$. By the pigeonhole principle, if $n$ is large enough, then there exist distinct $i,i'\in \{0,\dots,n/K-1\}$ such that \[(\mbf{v}_{iK},\mbf{v}_{iK+1},\dots,\mbf{v}_{(i+1)K-1})=(\mbf{v}_{i'K},\mbf{v}_{i'K+1},\dots,\mbf{v}_{(i'+1)K-1}).\] We may assume that $i=0$ and $i'=1$.

We will show that $\mbf{v}_0,\dots,\mbf{v}_{K-1}$ are pairwise distinct and that $\mbf{v}_i\cdot\mbf{v}_j$ is constant as $i,j\in\{0,\dots,K-1\}$ vary. We claim that these conditions force a contradiction. Indeed, in the case where $\mbf{v}_i\cdot\mbf{v}_j=0$ for all $i,j$, we have that the $\mbf{v}_i$ are indicator vectors of pairwise distinct subsets of $[2k-1]$ which each have even size, and each pair of which have even intersection. By Eventown, every such collection has at most $2^{(2k-1)/2}<2^k=K$ members, giving the required contradiction. On the other hand, if $\mbf{v}_i\cdot\mbf{v}_j=1$ for all $i,j$, then consider the `complement' vectors $\mbf{w}_0,\dots,\mbf{w}_{K-1}$, which have $1$'s where the $\mbf{v}_i$ have $0$'s and $0$'s where the $\mbf{v}_i$ have $1$'s. It is straightforward to use the fact that the vectors have odd length to show that these $\mbf{w}_i$ are pairwise distinct and satisfy $\mbf{w}_i\cdot\mbf{w}_j=0$ for all $i,j$, from which we can derive a contradiction as above.

We continue by showing that the vectors $\mbf{v}_0,\dots,\mbf{v}_{K-1}$ are pairwise distinct, which is equivalent to showing that each of $v_0,\dots,v_{K-1}$ is in a different atom. Suppose for a contradiction that $v_i$ and $v_j$ are in the same atom for some $i<j$ in $\{0,\dots,K-1\}$, and note that $v_{K+i}$ and $v_{K+j}$ are in this atom too by assumption. By the construction of $D$ there is some $\ell\in\{0,\dots,k-1\}$ such that $v_i$ and $v_j$ differ on $u_\ell$ in $D$. Since they are in the same atom as each other, they also differ on $u_\ell$ in $D'$. If $\vv{v_iu_\ell},\vv{u_\ell v_j}\in E(D')$, then to avoid a cyclic triple in $D'$ we have $\vv{v_iv_j}\in E(D')$. This means that the edges within $v_i$ and $v_j$'s atom have the same orientations in $D$ as they do in $D'$, so in particular we have $\vv{v_jv_{K+i}}\in E(D')$. Moreover $v_i$ and $v_{K+i}$ are in the same atom and are either both in-neighbours of $u_\ell$ in $D$ or both out-neighbours, so since $\vv{v_iu_\ell}\in E(D')$ we also have $\vv{v_{K+i}u_\ell}\in E(D')$. Hence, the cycle $\vv{v_jv_{K+i}u_\ell}$ appears in $D'$. Similarly if $\vv{v_ju_\ell},\vv{u_\ell v_i}\in E(D')$, then we have $\vv{v_jv_i}\in E(D')$. In this case the edges within the atom of $v_i$ and $v_j$ switch orientation between $D$ and $D'$, so the cycle $\vv{v_ju_\ell v_{K+i}}$ appears in $D'$. In both cases we have the desired contradiction, and we deduce that the vertices $v_0,\dots,v_{K-1}$ are all in different atoms.

It remains to show that $\mbf{v}_i\cdot\mbf{v}_j$ is constant as $i,j\in\{0,\dots,K-1\}$ vary. Suppose for a contradiction that this is not the case, then there exists $i\in\{0,\dots,K-1\}$ such that $\mbf{v}_i\cdot\mbf{v}_j$ is not constant as $j\in\{0,\dots,K-1\}$ varies. For such $i$ we can pick $j\in\{0,\dots,K-1\}$ such that $\mbf{v}_i\cdot\mbf{v}_i\neq\mbf{v}_i\cdot\mbf{v}_j$. Now if $\mbf{v}_i\cdot\mbf{v}_i=0$, then $\mbf{v}_i\cdot\mbf{v}_j=1$ so by \cref{obs:atoms}, $D'$ contains the cycle $\vv{v_{K+i}v_jv_i}$ if $i<j$ or the cycle $\vv{v_{K+i}v_{K+j}v_i}$ if $i>j$. Similarly if $\mbf{v}_i\cdot\mbf{v}_i=1$, then $D'$ contains one of the cycles $\vv{v_{i}v_jv_{K+i}}$ and $\vv{v_iv_{K+j}v_{K+i}}$. We have a contradiction in all cases, so the value of $\mbf{v}_i\cdot\mbf{v}_j$ is constant as $i,j\in\{0,\dots,K-1\}$ vary, as required.
\end{proof}

\section{Bounds on \texorpdfstring{$\inv(n)$}{inv(n)}}\label{sec:inv(n)}

\subsection{Lower bounds}\label{subsec:lower_bounds}

In this section we discuss the previous known lower bound on $\inv(n)$ and give the proof of \cref{thm:inv_lower}. As noted in the introduction, Belkhechine, Bouaziz, Boudabbous, and Pouzet~\cite{BELKHECHINE2010703} used a counting argument to lower bound $\inv(n)$. They observed that since there are $n!$ labelled transitive tournaments on $n$ vertices, there are at most $n!\cdot2^{n(k-1)}$ labelled $(k-1)$-invertible tournaments on $n$ vertices. There are a total of $2^{n(n-1)/2}$ labelled $n$-vertex tournaments, so for any $k$ such that $2^{n(n-1)/2}>n!\cdot2^{n(k-1)}$ we have $\inv(n)\geq k$. Taking logarithms base 2 and rearranging, this condition becomes $k<(n-1)/2-\log(n!)/n$, so we have
\[\inv(n)\geq \floor{\frac{n-1}{2}-\frac{\log(n!)}{n}}\geq \floor{\frac{n-1}{2}-\log(n)},\]
where for the final inequality we used $n!\leq n^n$. Lower bounds on $\inv(n)$ of this form were the best known (disregarding very slight tightenings of the argument).

The proof of \cref{thm:inv_lower} uses the following lemma which gives a bound on the probability that a random symmetric binary matrix has at most a certain rank.   
In fact, these probabilities are known exactly~\cite{Mac}, but we will use a simpler bound which is essentially tight for our purposes and for which we include a short proof.

\begin{lemma}\label{lem:inv_lower}
The probability that a uniformly random
$n\times n$ symmetric matrix over $\F_2$ has rank at most
$n-s$ (over $\F_2$) is at most $2^{s \log(n)-\binom{s}{2}}$.
\end{lemma}
\begin{proof}
Construct the random matrix in $n$ steps, in the $i$th step choosing the first $i$ entries of the $i$th row of the matrix (and also, by symmetry, the $i$th column). For each $i\in[n]$, let $M_i$ be the random symmetric $i\times i$ matrix obtained after step $i$.

Note that for each $i$ the nullity increases by at most 1 between $M_i$ and $M_{i+1}$. It follows that if the nullity of $M_n$ is at least $s$, then for all $1 \leq j\leq s-1$ we can define $k_j$ to be the smallest $i$ such that the nullity of $M_i$ is $j+1$, and we have $2\leq k_1<k_2<\dots<k_{s-1}\leq n$.
For each $j$, the ranks of $M_{k_j-1}$ and $M_{k_j}$ are equal, so the first $k_j-1$ entries of the $k_j$th row of $M_{k_j}$ lie in the $(k_j-1-j)$-dimensional row space of $M_{k_j-1}$, which happens with probability $2^{-j}$. There are $\binom{n}{s-1}$ ways to choose $k_1,\dots,k_{s-1}$ as above, so the probability that $M_n$ has rank at most $n-s$ is at most
\[
\binom{n}{s-1}\prod_{j=1}^{s-1} 2^{-j} \leq  2^{s \log(n)-\binom{s}{2}},
\]
as required.
\end{proof}

We are now ready to prove \cref{thm:inv_lower}.

\begin{proof}[Proof of \cref{thm:inv_lower}]
Let $T$ be a uniformly random tournament on vertex set
$[n]$ and let $M_T=(m_{ab})$ be the $n \times n$ matrix over $\F_2$ defined as follows.
For $a<b$, let $m_{ab}$ be $0$ if $\vv{ab}$ is an edge of $T$ and 
$1$ otherwise, then define $m_{ba}=m_{ab}$, and finally choose each 
diagonal entry uniformly at random. Note that the $\binom{n}{2}$ entries of $M_T$ above the diagonal determine $T$, and the other entries are defined such that $M_T$ is a uniformly random symmetric binary matrix.

Let $s=\floor{\sqrt{2 n \log(n)}}$ and write $k=n-s$. Suppose that $\inv(T)\leq k$ and let $X_1, \ldots ,X_k$ be a decycling family of $T$. For each $X_i$, let $M_i$ be the $n\times n$ binary matrix whose $(a,b)$ entry is $1$ if and only if $a,b\in X_i$. Observe that, working over $\F_2$, we have $\rk(M_i)\leq 1$ for all $i$, and thus $\rk(\sum_i M_i) \leq k$. 
By construction, $M_T+\sum_i M_i$ is a matrix whose entries above the diagonal correspond to a transitive tournament on $[n]$ (its diagonal entries can be anything). Let $\cM$ be the set of binary matrices corresponding in this manner to a transitive tournament on $[n]$, and note that $\abs{\cM}=n!2^n$.

Putting all of this together, we have that if $\inv(T)\leq k$, then there exists $M\in\cM$ such that $\rk(M_T+M)\leq k$. For each fixed $M$, we have that $M_T+M$ is a uniformly random symmetric binary matrix and hence has rank at most $k$ with probability at most $2^{s \log(n)-\binom{s}{2}}$ by \cref{lem:inv_lower}.
Taking a union bound over all $M \in \cM$ we obtain
\[
\bP(\inv(T)\leq k)  \leq n! 2^n 2^{s \log(n) -\binom{s}{2}}.
\]
Since $n! =O(\sqrt{n}(n/e)^n)$, the right-hand side is $O(2^{f(n)})$ where
\begin{align*}
f(n)&= \frac{\log(n)}{2} + n \log(n) - n \log(e) + n + s \log(n) - \binom{s}{2}
\\ &= - n (\log(e)-1) +o(n),
\end{align*}
 and thus $\bP(\inv(T)\leq k) \to 0$ as $n \to \infty$
as desired.
\end{proof}

\subsection{Upper bounds}\label{subsec:upper_bounds}

The only approach which has been used to prove upper bounds on $\inv(n)$, introduced in~\cite{BELKHECHINE2010703}, is to `solve' one vertex at a time, as follows. Given a tournament $T$, pick a vertex $v$ and invert the set consisting of $v$ and its out-neighbourhood. In the resulting tournament $T_1$, $v$ is a sink. Using a further $\inv(n-1)$ inversions we can transform $T_1-\{v\}$ into a transitive tournament, so $\inv(n)\leq \inv(n-1)+1$ for all $n\geq 2$. The authors of~\cite{BELKHECHINE2010703} observed that $\inv(4)=1$, so $\inv(n)\leq n-3$ for $n\geq 4$. For $n\geq 6$ this was improved by~1 in~\cite{bangjensen2022inversion} using the fact that $\inv(6)= 2$ (which they attribute to~\cite{BELKHECHINE_unpublished} and which we have verified by a computer check). We introduce a slightly different approach to prove the following.
\begin{prop}\label{prop:inv_ub}
For all $n\in\N$, \[\inv(n)\leq \floor{\frac{n-1}{2}}+\inv\left(\ceil{\frac{n-1}{2}}\right).\]
\end{prop}
\begin{proof}
Let $n\in\N$ and let $T$ be an $n$-vertex tournament. Pick $v\in V(T)$ and write $A$ and $B$ for the in- and out-neighbourhoods of $v$ respectively. We may assume that $\abs{A}\geq \ceil{(n-1)/2}$ (the case where $B$ is the larger of the two is similar). By `solving' each vertex in $B$ one after another, we can find at most $\abs{B}$ inversions which transform $T$ into a tournament $T'$ such that the subtournament of $T'$ induced on $B\cup\{v\}$ is transitive (with $v$ as the minimal element) and every edge of $T'$ between $A$ and $B\cup\{v\}$ is oriented away from $A$. With a further $\inv(A)\leq \inv(\abs{A})$ inversions we can transform $T'$ into a transitive tournament. Thus, $\inv(T)\leq \abs{B}+\inv(\abs{A})$.

We have $\inv(k)\leq \inv(k-1)+1$ for all $k\in\N$ and we can apply this $\abs{A}-\ceil{(n-1)/2}$ times to obtain $\inv(\abs{A})\leq \inv(\ceil{(n-1)/2})+\abs{A}-\ceil{(n-1)/2}$. Using the fact that $\abs{A}+\abs{B}=n-1$, this yields \[\inv(T)\leq \abs{B}+\inv\left(\ceil{\frac{n-1}{2}}\right)+\abs{A}-\ceil{\frac{n-1}{2}}=\floor{\frac{n-1}{2}}+\inv\left(\ceil{\frac{n-1}{2}}\right),\] and the claim follows.
\end{proof}

We can use this result to improve (for large $n$) the upper bound on $\inv(n)$.

\begin{corollary}\label{cor:inv_ub}
For all $n\in\N_0$, $\inv(n)\leq n-\log(n+1)$.
\end{corollary}
\begin{proof}
We prove the statement by induction on $n$, with the case $n=0$ clear. If $n\geq 1$ and the claim holds for all smaller values, then we have
\begin{equation*}
    \begin{split}
        \inv(n)
        & \leq \floor{\frac{n-1}{2}}+\inv\left(\ceil{\frac{n-1}{2}}\right),\\
        & \leq \floor{\frac{n-1}{2}} + \ceil{\frac{n-1}{2}}-\log\left(\ceil{\frac{n-1}{2}}+1\right), \\
        & \leq n-1-\log\left(\frac{n+1}{2}\right),\\
        & = n-\log(n+1).
    \end{split}
\end{equation*}
\end{proof}

\section{Conclusion}\label{sec:conclusion}

In this paper we have answered several of the questions posed in~\cite{bangjensen2022inversion}. We have shown that their `dijoin conjecture', that $\inv(L\ra R)=\inv(L)+\inv(R)$, is false in general, but have verified it in the case where $\inv(L)=\inv(R)=2$ and have also shown that a $k$-join analogue holds under certain conditions. In addition, we have confirmed their related conjectures that \textsc{$k$-Inversion} is NP-complete for all $k\geq 1$, and that the inequality $\inv(D)\leq 2\tau(D)$ is tight. We have answered their question concerning the minimal $r_k$ such that \textsc{$k$-Tournament-Inversion} can be solved in time $O(\abs{V(T)}^{r_k})$, showing that $r_k=2$ for all $k$, and have improved the lower bound on $\inv(n)$ to show that $\inv(n)=(1+o(1))n$. There are, however, still many interesting open problems in this area. Before discussing some of them, we will touch on two operations similar to inversion.

\subsection{Similar operations}\label{subsec:similar_operations}

We first consider an operation on permutations which is used by molecular biologists as a model for genetic mutations, and could loosely be seen as a vertex analogue of inversions in tournaments. Given a permutation $\pi=(\pi_1\ \pi_2\ \ldots\ \pi_n)$ of $[n]$, for $1\leq i<j\leq n$, the \emph{reversal of the interval $[i,j]$} is the permutation obtained by reversing the order of $\pi_i,\dots,\pi_j$ in $\pi$.  The \emph{reversal distance}, $d(\pi)$, of a permutation $\pi$ is the minimum number of reversals required to transform $\pi$ into the identity permutation. For a survey of reversals and the reversal distance (and many other combinatorial models of genome rearrangements) see~\cite{biobook}. We highlight some results of particular relevance to our work. With regards to computational complexity, Caprara~\cite{caprara} showed that the problem of \textsc{Sorting by Reversals}, that is, determining whether $d(\pi)\leq k$ for inputs of a permutation $\pi$ and $k\in\N$, is NP-complete, while Hannenhalli and Pevzner~\cite{HP96,HP99} showed that it is fixed-parameter tractable when parameterised by~$k$. The natural extremal problem was solved by Bafna and Pevzner~\cite{bafna1996genome} who proved that for a permutation $\pi$ of $[n]$, we have $d(\pi)\leq n-1$ with equality if and only if $\pi\in\{\gamma_n,\gamma_n^{-1}\}$ for an explicit~$\gamma_n$. 

Inversions in digraphs can also be thought of as generalisations of \emph{edge reversals}, i.e.\ the operations which reverse the orientation of a single edge. It is not difficult to see (using an argument from \cref{subsec:intro_tau}) that the minimum number of such operations required to transform a digraph $D$ into an acyclic digraph is equal to $\tau'(D)$, the cycle edge-transversal number of $D$. Determining this quantity is the famous feedback arc set problem, which has been widely studied (see~\cite{kudelic} for an overview). In particular the problem of determining for inputs $D$ and $k$ whether $\tau'(D)\leq k$ was one of the first shown to be NP-complete~\cite{karp} and it remains NP-complete when the input is restricted to tournaments~\cite{alon,charbit}. However, Chen, Liu, Lu, O'Sullivan, and Razgon \cite{chen2008fixed} showed that this problem is again fixed-parameter tractable when parameterised by~$k$. On the extremal side, it was shown by Spencer~\cite{spencer1971optimal, spencer1980optimally} that the maximum cycle edge-transversal number of an $n$-vertex tournament is $\frac{1}{2}\binom{n}{2}-\Theta(n^{3/2})$ and that a random labelled $n$-vertex tournament has this cycle edge-transversal number with probability tending to 1. Bounds of this form remain the best known (see also~\cite{de1983maximum,poljak1988tournament}).

\subsection{Open problems}\label{subsec:open_probs}
We have shown (in \cref{thm:tournament_fpt}) that the problem which takes as inputs a tournament $T$ and an integer $k\in\N$, and asks whether $\inv(T)\leq k$, is fixed-parameter tractable when parameterised by $k$. In keeping with the pattern exhibited in the settings discussed in \cref{subsec:similar_operations}, Bang-Jensen, da Silva, and Havet~\cite{bangjensen2022inversion} conjectured that the full problem is NP-complete.

\begin{conjecture}[\cite{bangjensen2022inversion}]\label{conj:tnmnt_NPcomplete}
The problem of deciding whether $\inv(T)\leq k$ for inputs of $k\in \N$ and a tournament $T$ is NP-complete.
\end{conjecture}

Note that \cref{thm:tournament_fpt} does not make any progress towards disproving this because the implied constant in the $O(n^2)$ running time is not polynomial in $k$. In fact, as noted above, the constant arising from our algorithm is doubly exponential in $k$. However, again in keeping with both settings discussed in \cref{subsec:similar_operations} (and indeed many natural fixed-parameter tractable problems), we conjecture that this constant can be taken to be singly exponential in $k$, perhaps with a higher power of $n$.

\begin{conjecture}\label{conj:fpt_improved}
 There exist constants $c_1,c_2> 0$ such that \textsc{$k$-Tournament-Inversion} can be solved in time $O(2^{k^{c_1}}\abs{V(T)}^{c_2})$ for any $k\in\N$.
\end{conjecture}

As discussed in the introduction, the set $\mathcal{IC}_k$ of $k$-inversion-critical tournaments was shown to be finite for all $k$ in~\cite{BELKHECHINE2010703}. They explicitly described $\mathcal{IC}_1$ and $\mathcal{IC}_2$, for the latter using results of Gallai~\cite{Gallai67} (see~\cite{Gallai67Eng} for an English translation) and Latka~\cite{latka03}, but for $k\geq 3$ very little is known about these sets. In particular, it would be interesting to determine $m_k$, the maximum number of vertices in a tournament in $\mathcal{IC}_k$, for $k\geq 3$.

\begin{question}[\cite{bangjensen2022inversion}]\label{q:m_k}
What is the value of $m_k$ for $k\geq 3$?
\end{question}

Finding the minimum possible size of a $k$-inversion-critical tournament is equivalent to the problem of determining $\inv(n)$. The best known bounds on $\inv(n)$ for large $n$ are now
\[ n-\sqrt{2n\log(n)}\leq \inv(n)\leq n-\log(n+1),\]
and it would be interesting to tighten these further.

\begin{question}\label{q:inv(n)}
What is the asymptotic behaviour of $n-\inv(n)$?
\end{question}

In light of our improved lower bound on $\inv(n)$, the lack of an explicit construction for a tournament of large inversion number is even more apparent: no $n$-vertex construction with inversion number more than about $n/3$ (as given by the $(n/3)$-join $[\vv{C_3}]_{n/3}$) is known. 

\begin{problem}\label{prob:construct_inv(n)}
Construct $n$-vertex tournaments with inversion number closer to $\inv(n)$.
\end{problem}

Belkhechine, Bouaziz, Boudabbous, and Pouzet (\cite{BELKHECHINE_unpublished}; see~\cite{bangjensen2022inversion}) defined for each $n\in\N$ a tournament $Q_n$ on vertex set $[n]$ in which for $i<j$ the edge $ij$ is oriented towards~$j$, except if $j=i+1$, in which case it is oriented towards $i$, and conjectured that these graphs satisfy $\inv(Q_n)=\floor{\frac{n-1}{2}}$.

\begin{conjecture}[\cite{BELKHECHINE_unpublished}]\label{conj:inv_lb_example}
For all $n\in \N$ we have $\inv(Q_n)=\floor{\frac{n-1}{2}}$.
\end{conjecture}

The conjecture is known to hold for $n\leq 8$ \cite{BELKHECHINE2010703,bangjensen2022inversion}, and it is certainly true that $\inv(Q_n)\leq \floor{\frac{n-1}{2}}$ for all $n$ since the sets \[\{2,3\},\{4,5\},\{6,7\},\dots,\{2\floor{(n-1)/2},2\floor{(n-1)/2}+1\}\] form a decycling family of $Q_n$. 

Defining the \emph{inversion distance}, $\inv(T,T')$, between two labelled tournaments $T$ and $T'$ on the same vertex set to be the minimum number of inversions required to transform $T$ into $T'$, we remark that the matrix rank techniques developed in \cref{subsec:lower_bounds} can be used to show that the maximum inversion distance between two $n$-vertex tournaments is exactly $n-1$. Moreover, combining these ideas with \cref{lem:inv_lower} gives an upper bound of $2^{\binom{n}{2}+n- \binom{s}{2}+s \log(n)}$ on the number of labelled tournaments within inversion distance $n-s$ of a given labelled tournament. 

It is natural in this context to study the random walk $\cW$ on the space of labelled tournaments on $[n]$ where each step in the walk consists of picking a uniform random subset of $[n]$ and inverting that set in the current tournament. In particular, we ask the following.

\begin{question}\label{q:mixing}
What is the mixing time of $\cW$? Does it satisfy the cutoff phenomenon?
\end{question}

Returning to the dijoin conjecture, \cref{thm:2ra2} completes the work of Bang-Jensen, da Silva, and Havet in showing that the conjecture holds in the cases where $\inv(L),\inv(R)\leq 2$. We have also shown (\cref{thm:inv_kjoin}) a $k$-join analogue of the dijoin conjecture for collections of 2-invertible digraphs $D_1,\dots,D_k$ at most one of which has inversion number 2. We conjecture that this final condition can be removed.

\begin{conjecture}\label{conj:kjoin_true}
Let $k\in\N$ and let $D_1,\dots,D_k$ be oriented graphs satisfying $\inv(D_i)\leq 2$ for all~$i$. Then \[\inv([D_1,\dots,D_k])=\sum_{i=1}^k \inv(D_i).\]
\end{conjecture}

On the other hand, \cref{thm:dijoin_ce} gives a counterexample to the dijoin conjecture where $\inv(L)=1$ and $\inv(R)=3$. From this, we can obtain counterexamples with $\inv(L)=k$ and $\inv(R)=3$ for any $k\in\N$: let $L=[\vv{C_3}]_k$ and let $R$ be as in the proof of \cref{thm:dijoin_ce}. The tournaments obtained from these by inverting the whole vertex set give counterexamples in which $\inv(L)=3$ and $\inv(R)=k$. We conjecture that here $3$ can be replaced with any larger integer, or in other words that the only values of $\inv(L)$ and $\inv(R)$ for which the dijoin conjecture always holds are those where $\inv(L),\inv(R)\leq 2$ or where one of $\inv(L)$ or $\inv(R)$ is $0$.

\begin{conjecture}\label{conj:dijoin_false}
For all $\ell,r\in\N$ with $\ell\geq 3$ or $r\geq 3$ there exist oriented graphs $L$ and $R$ with $\inv(L)=\ell$ and $\inv(R)=r$, but $\inv(L\ra R)<\ell+r$.
\end{conjecture}

This conjecture is equivalent to the claim that for all $r\geq 3$ there exists a tournament~$R$ with $\inv(R)=\inv(\vv{C_3}\ra R)=r$. To see that this follows from the conjecture, note that for $r\geq 3$, if $\inv(L)=1$ and $\inv(R)=\inv(L\ra R)=r$, then we can extend $L\ra R$ to a tournament $T=L' \ra R'$ with inversion number $r$. Clearly $\inv(R')=r$ and $\inv(L')\geq 1$, so $L'$ contains a copy of $\vv{C_3}$. Thus, $r=\inv(R')\leq\inv(\vv{C_3}\ra R')\leq \inv(T)=r$, as required. The converse follows from the arguments of the previous paragraph.

Finally, we noted in \cref{subsec:intro_complexity} that $\inv(D)\leq \inv(D-\{v\})+2$ for all digraphs~$D$ and vertices $v\in V(D)$. It is certainly the case that this inequality is tight for some $D$ and $v$. Indeed, a reformulation of \cref{thm:tau} yields the stronger statement that for all $k\in\N$ there exists a digraph $D$ and a set $S\subseteq V(D)$ with $\abs{S}=k$ such that for all $T\subseteq S$ we have $\inv(D-T)=\inv(D)-2\abs{T}$. We conjecture, however, that the inequality $\inv(D)\leq \inv(D-\{v\})+2$ cannot be tight for all vertices $v$ in a given digraph $D$.

\begin{conjecture}\label{conj:vertex_ineq}
Let $D$ be a digraph with at least one vertex. Then there exists $v\in V(D)$ such that $\inv(D-\{v\})\geq \inv(D)-1$.
\end{conjecture}

\section*{Acknowledgments}
We would like to thank the anonymous referees whose suggestions improved the presentation and clarity of our arguments. We would also particularly like to thank one referee for raising \cref{conj:fpt_improved} as a question.

\end{document}